\documentclass[11pt]{amsart}
\usepackage[margin=2.4cm]{geometry}
\usepackage{amsmath}
\usepackage{amsfonts}
\usepackage{amssymb}
\usepackage{xcolor}
\usepackage{url}
\renewenvironment{proof}{{\bfseries Proof.}}{\qed}
\numberwithin{equation}{section} 
\newtheorem{theorem}{Theorem}[section] 
 
\newtheorem{cor}[theorem]{Corollary} 
\newtheorem{lemma}[theorem]{Lemma} 
\newtheorem{thm}[theorem]{Theorem}
\newtheorem{defn}[theorem]{Definition} 
\newtheorem{rk}[theorem]{Remark} 
\newcommand{\gt}{\tilde{\gamma}}

\newcommand{\n}{\newline}

\newcommand{\tb}{\textbf}

\newcommand{\imp}{\Rightarrow}
\newcommand{\z}{\mathbb{Z}}

\newcommand{\g}{\gamma_{p}}

\newcommand{\io}{\iota }
\newcommand{\secref}[1]{Section~\ref{#1}}
\newcommand{\thmref}[1]{Theorem~\ref{#1}}
\newcommand{\lemref}[1]{Lemma~\ref{#1}}

\newcommand{\corref}[1]{Corollary~\ref{#1}}
\newcommand{\eqnref}[1]{~{\textrm(\ref{#1})}}

\numberwithin{equation}{section}

\makeatletter
\@namedef{subjclassname@2020}{\textup{2020} Mathematics Subject Classification}
\makeatother

\date{\today}

\begin{document}
	\title[Asymptotic growth of the number of Reciprocal Classes in the Hecke Groups]{ Asymptotic growth of the number of Reciprocal Classes in the Hecke Groups }
	\author{Debattam Das  \and Krishnendu Gongopadhyay}
    
	\address{Indian Institute of Science Education and Research (IISER) Mohali,
		Knowledge City,  Sector 81, S.A.S. Nagar 140306, Punjab, India}
	\email{debattam123@gmail.com}
    \address{Indian Institute of Science Education and Research (IISER) Mohali,
		Knowledge City,  Sector 81, S.A.S. Nagar 140306, Punjab, India}
	\email{ krishnendu@iisermohali.ac.in}

	\begin{abstract}
    We estimate the asymptotic growth of reciprocal conjugacy classes in Hecke groups using their free product structure and word lengths of reciprocal elements. Our approach is different from other works in this direction and uses tools from basic probability theory.
	\end{abstract}
	
	\maketitle	
	
	\section{Introduction}\label{1}

The \emph{Hecke group} \( \Gamma_p \) is a Fuchsian group isomorphic to the free product \( \mathbb{Z}_2 * \mathbb{Z}_p \). It is generated by the maps:
\[
\iota: z \mapsto -\dfrac{1}{z}, \qquad \alpha_p: z \mapsto z + \lambda_p, \quad \text{with} \quad \lambda_p = 2 \cos\left(\frac{\pi}{p}\right), \quad p \geq 3.
\]
In particular, \( \Gamma_3 \) corresponds to the modular group. The composition \( \iota \alpha_p : z \mapsto -\dfrac{1}{z + \lambda_p} \) has order \( p \). Let \( \gamma = \iota \alpha_p \), so that \( \Gamma_p \) has the presentation:
\[
\langle \iota, \gamma \mid \iota^2 = \gamma^p = 1 \rangle.
\]
This presentation will be fixed throughout the paper. When \( p = 2r \), the element \( \gamma^r \) is an involution, denoted by \( \tilde{\gamma} = \gamma^r \).

 Recall that a non-trivial element \( g \) in a group \( \Gamma \) is called \emph{reversible} (or \emph{real}) if there exists \( h \in \Gamma \) such that \( hgh^{-1} = g^{-1} \). If \( h \) is an involution (i.e., \( h^2 = 1 \)), then \( g \) is \emph{strongly reversible} (or \emph{strongly real}). Every strongly reversible element other than involutions can be written as a product of two involutions. For more details on reversibility from a diverse viewpoint, see, e.g. \cite{BR, Sa, OS}. For torsion-free elements in a Fuchsian group $\Gamma$, a reversible element is always strongly reversible, and moreover, any element \( h \) satisfying \( hgh^{-1} = g^{-1} \) must be an involution. In the modular group \( \mathrm{PSL}(2, \mathbb{Z}) \), such elements are named \emph{reciprocal} due to 
their connections with the reciprocal integral binary quadratic forms of Gauss, see \cite{Sa}, \cite{bppz}. We shall also use the same terminology for Hecke groups. A conjugacy class containing a reciprocal element is called a \emph{reciprocal class}. When $g$ is hyperbolic and reciprocal, the subgroup generated by $g$ and a conjugating involution $h$ forms an infinite dihedral subgroup. Consequently, classifying reciprocal classes corresponds to classifying infinite dihedral subgroups up to conjugacy. The reciprocal classes also correspond to the closed geodesics of the corresponding hyperbolic orbifold $\mathbb{H}^2\slash \Gamma$ that passes through an even-ordered cone point. Such geodesics are called `recipocal geodesics'. Thus,  by counting reciprocal classes, we count reciprocal geodesics in $ \mathbb{H}^2\slash \Gamma$.

\medskip In this paper, we introduce a new approach to estimate the asymptotic growth of the number of  reciprocal classes in the Hecke group \( \Gamma_p \), valid for arbitrary \( p \). This method diverges from all prior works in this direction and is based on the Central Limit Theorem for normal  probability distributions. We use the free product structure for the Hecke groups. The  \emph{word length} of the conjugacy class of a word is the minimum length of a cyclically reduced word present in that conjugacy class. After establishing a \emph{normal form} for reciprocal words, we apply probabilistic techniques to prove our results, given below.

For a real number \( x \), let \( \lfloor x \rfloor \) denote the greatest integer less than or equal to \( x \), and \( \lceil x \rceil \) the smallest integer greater than or equal to \( x \).

We define $f \precsim g
$ for real-valued functions $f,g$ on $[0,\infty)$ if there exists a constant $C>0$ and $t_0$ such that $f(t)\leq Cg(t)$ for all $t\geq t_0.$
 We denote $f\simeq g $  if $f\precsim g$ and $g\precsim f$. Also, we denote $ f\sim g$ if the ratio of $f(t)$ and $g(t)$ tends to $1$, as $t\to \infty.$ Note that $f\sim g$ is a stronger condition than $f\simeq g.$
 Further, let $\Phi$ be the cumulative distribution function of the standard normal distribution $N(0,1)$, then $\Phi(x)>0$ for all finite $x\in \mathbb{R}$.  With these notations, we now state our main theorems.

		\begin{thm}\label{main1}
		For odd $p=2r+1 $, let $\mathcal{Q}_{x}$ be the set of reciprocal classes in the Hecke group $\Gamma_{p}$, which have at most $x$ word lengths. Then,
		\[|\mathcal{Q}_{x}|\sim 	\frac{1}{2}\sum_{n=1}^{\lfloor \frac{x}{2(r+1)}\rfloor}(2r)^{n}+\frac{1}{2}{\sum_{n=\lceil \frac{x}{2(r+1)}\rceil}^{\lfloor\frac{x}{4}\rfloor}(2r)^{n}\Phi\left(\frac{x-n\mu}{\sqrt{n}\sigma}\right)} \] where, $\mu=r+3$ and $\sigma^{2}=\frac{r^{2}-1}{3}.$
	\end{thm}
	
 \begin{thm}\label{main2}
	For even $p=2r $, let $Sym_x$ be the set of symmetric reciprocal classes and  $T_{x}$ be the set of  all reciprocal classes respectively in the Hecke group $\Gamma_{p}$ with word length at most $x$.  Then,
	\[|T_x|\simeq 	\frac{1}{2}\sum_{n=1}^{\lfloor \frac{x}{2(r+1)}\rfloor}(2r-1)^{n}+\frac{1}{2}{\sum_{n=\lceil \frac{x}{2(r+1)}\rceil}^{\lfloor\frac{x}{4}\rfloor}(2r-1)^{n}\Phi\left(\frac{x-n\mu}{\sqrt{n}\sigma}\right)} \] 

    where, $\mu=\frac{2r^2 + 4r - 2}{2r - 1}$ and $\sigma^{2}=\frac{16r^3 + 36r^2 + 32r - 12}{6(2r - 1)}.$
\end{thm}	\medskip

An element $a$ of a group $G$ is said to be a primitive element if it cannot be presented as $b^k$ for some element $b\in G$ and $k\in \z$. In the case of Hecke groups, the primitive hyperbolic elements are invariant under conjugation, and they are responsible for getting primitive geodesics in the orbit space of the Hecke group.

In the following corollary, we have shown that the asymptotic growth of the number of reciprocal classes is the same as the asymptotic growth of the number of primitive reciprocal classes in Hecke groups. 
 \begin{cor}\label{cormain}
     Let $W^p_x$ and $W_x$ be the collection of primitive reciprocal classes, and the collection of reciprocal classes, resp., present in the Hecke group $\Gamma_p$. Then, we have
     \[W^p_x\sim W_x.\]
 \end{cor}
	As a corollary to the first theorem, we recover the following result. This was proven by  Basmajian and Suzzi-Valli in \cite[Theorem 1.1]{BV}. 
    \begin{cor}\label{1.4}
    Let $\left|\gamma\right|$ be the minimal word length in the conjugacy
	          	class of a lift of a reciprocal geodesic $\gamma$ in $\rm PSL(2,\z)\cong \Gamma_3$. Then
	          	$$\left|~\gamma \text{ a primitive reciprocal geodesic with }\left|\gamma\right|\leq 2t~\right|\sim 2^{\lfloor \frac{t}{2}\rfloor}. $$

	          \end{cor}
    
    \medspace

We now remark about known works in this direction.  The growth and distribution of geodesics representing reciprocal classes have been studied from various viewpoints.  In \( \mathrm{PSL}(2, \mathbb{Z}) \), Sarnak \cite{Sa} classified, counted and discussed equidistribution results for the reciprocal classes in the modular group, also see the works \cite{BK}, \cite{bv2}, \cite{bl} on related directions.  Erlandsson and Souto \cite{ES} generalized Sarnak's counting result for reciprocal classes in a lattice of ${\rm PSL}(2,\mathbb R)$ that contains an involution. These studies have used the hyperbolic length as a tool for the asymptotic count. Recently, Parkkonen and Paulin \cite{PP} have introduced a thermodynamic formalism to further generalize the earlier results.

An alternative approach, based on the free product structure of the modular group, was developed by Basmajian and Suzzi Valli \cite{BV}, who analyzed reciprocal class growth via combinatorial word lengths in the modular group. Marmolejo \cite{Ma} extended this to Hecke groups \( \Gamma_p \) for odd \( p \), determining the \emph{spherical growth}—the number of reciprocal classes represented by words of a fixed length. It has been extended to all Hecke groups in \cite{DG2}, also see \cite{bmv}.

	 Now, we briefly mention the structure of this paper. After recalling some preliminary notions, we note preliminary observations about the above counting function in \secref{prel}. Finally, we prove the above theorems in \secref{4}.

	\section{Preliminaries}\label{prel}
	
	\subsection{Reciprocal elements in Hecke groups}In \cite{DG}, different types of reciprocal elements for Hecke groups were defined. Here the following types are:
	\begin{defn}
		The reciprocal word is said to be \textit{symmetric} if it can be seen as a product of those elements which are only conjugate to $ \io $.
	\end{defn}
	\begin{defn}
		The reciprocal word is said to be \textit{$ p $-reciprocal} if it can be seen as a product of those elements which are only conjugate to $ \gt $.
	\end{defn}
	\begin{defn}
		The reciprocal word is said to be \textit{symmetric $ p $-reciprocal} if it can be seen as a product of conjugate of $\io$ and conjugate of $ \gt $ respectively.
	\end{defn}
	\begin{rk}
		If $p$ is an odd integer, all reciprocal elements are symmetric. Other types appear in the even case.
	\end{rk}
	\subsection{Word Length} In this section, we will define the word length of the elements of the free products and conjugacy classes.
	\begin{defn}\label{word}
		The\textit{ length} of a word is the number of the generating elements of the group presented in its reduced presentation.
	\end{defn}
	For instance, we consider the Hecke group $\Gamma_{p} =\langle a, b ~ | ~ a^{2},b^{p} \rangle  $ , where $p\geq 3$ , then the length of the reduced word $ ab^{k_{1}}ab^{k_{2}}\dots ab^{k_{n}} $ , where $-p/2\leq k_{i}\leq p/2 $, is $ \sum_{i=1}^{n}\left(1+|{k_{i}}|\right) $.
	\begin{defn}
		The \emph{length} of the conjugacy class of a word is the minimum length of a cyclically reduced word present in that conjugacy class.
	\end{defn}
    
	Before discussing on asymptotic growth of reciprocal classes, we need to recall the following results.
	\begin{theorem}{\cite{DG}}\label{th}
		For any reciprocal hyperbolic element in $ \Gamma_{p} $, the following possibilities can occur :
		\begin{enumerate}
			\item If p is odd, every reciprocal class has exactly four symmetric elements.
			\item If p is even.
			\begin{enumerate}
				\item Suppose in the reciprocal class, the reciprocators are conjugate to $ \iota $. Then, there are exactly four symmetric elements in that class.
				\item Suppose in the reciprocal class, the reciprocators conjugate to $\gamma_{p}^{r} $. Then, there are exactly $ 2p $ many $p$-reciprocal elements. 
				\item Suppose in the reciprocal class, each reciprocal element has two types of reciprocators; one type of reciprocator conjugates to $ \iota $, and another type of reciprocator conjugates to $ \gamma_{p}^{r} $.
				\begin{enumerate}
					
					\item If the reciprocal class does not contain any non-zero power of $ \iota \g^r$, then the class has exactly two symmetric elements and $ p $ many $p$-reciprocal elements.

					\item If the reciprocal class contain any non-zero power of $ \iota \g^r$, then the class has exactly two symmetric p-reciprocal elements and $p-2$  $p$-reciprocal elements.			\end{enumerate}
			\end{enumerate}
		\end{enumerate}
		
	\end{theorem}

\begin{lemma}\label{lem}\cite[Lemma 3.2]{DG2} Let $[{\tilde{\alpha}}]$ be a reciprocal class in $\Gamma_p$.  Then, we can choose a cyclically reduced representative $\alpha$ (or $\alpha^{-1}$) of this class as follows.  
		
		\begin{enumerate}
			\item Suppose all the reciprocators of $[\alpha]$ are conjugate to $\io$. Then  $\alpha$ is of the form  \\$ \io\g^{k_{1}}\io \g^{k_{2}}\dots \io \g^{k_{n}}\io \g^{-k_{n}}\dots\io \g^{-k_{2}}\io \g^{-k_{1}}  $.
			\item Suppose all the reciprocators of $[\alpha]$ are  conjugate to $ \gt $. Then $\alpha$ is of the form  \n  $ \io \gt\io \g^{k_{1}}\io \g^{k_{2}}\dots \io \g^{k_{n}}\io \gt\io \g^{-k_{n}}\dots\io \g^{-k_{2}}\io \g^{-k_{1}} $.
			\item Suppose  $[\alpha]$ has two types of reciprocators, some are conjugate to  $ \io $ and others conjugate to $ \gt $. 
			If the reciprocal class does not contain any power of $\io \gt$,   then $\alpha$ has the form  either \n
			$ \io\gt\io  \g^{k_{1}}\io \g^{k_{2}}\dots \io \g^{k_{n}}\io \g^{-k_{n}}\dots\io \g^{-k_{2}}\io \g^{-k_{1}} $ or $\io  \g^{k_{1}}\io \g^{k_{2}}\dots \io \g^{k_{n}}\io\gt\io \g^{-k_{n}}\dots\io \g^{-k_{2}}\io \g^{-k_{1}}$.

			If $[\alpha]$ contains any non-trivial power of $\io \gt  $, then it has a unique representative of the form $ (\io\gt )^{k}$ where $ k\in \z$. 
	\end{enumerate}\end{lemma}

		\section{Asymptotic Growth for the Reciprocal Classes}\label{4}
		From the results of \cite{DG2}, it is easy to obtain the exact asymptotic growth of the number of reciprocal classes for small integers \( n \), but not for large \( n \). Therefore, for large \( n \), we use the Central Limit Theorem to compute the asymptotic growth with respect to the word length \( x \) and the number \( n \) of \( k_i \)'s, where the \( k_i \)'s are as defined in \lemref{lem}.
 From now onwards, we consider $\Phi$ as the Cumulative distribution function, in short, cdf, of the normal distribution $N(0,1)$. Thus, we have the following lemma:
		 
		 \begin{thm}[\tb{Central Limit Theorem}]
		 	Let \( X_1, X_2, \dots, X_n \) be i.i.d. random variables with mean \( \mu \) and variance \( \sigma^2 \). Then
		 	\[
		 	 \left( \frac{\sum_{i=1}^n X_i-n \mu}{\sigma\sqrt{n}} \right) \xrightarrow{d} {N}(0, 1).
		 	\]
		 \end{thm}

		 To get the asymptotic value, we use the Central Limit Theorem to get the following lemma: 		\begin{lemma}\label{even}
		 	Let us consider the equation:

		 	\begin{equation}
		 		\sum_{i=1}^{n} |~p_{i}~|=S_{n}, 
		 	\end{equation}
		 	where, $p_{i}$'s, and $x$ are positive integers, and each $p_{i}$  satisfies \( -r < p_i \leq r \). Let $\mathcal{N}_x$ be the number of the solution of the mentioned equation subject to the condition $S_n \leq x$. Then we have  \[ \mathcal{N}_{x}\sim \sum_{n=1}^{\lfloor \frac{x}{r}\rfloor}(2r-1)^{n}+{\sum_{n=\lceil \frac{x}{r}\rceil}^{x}(2r-1)^{n}\Phi\left(\frac{x-n\mu}{\sqrt{n}\sigma}\right)},  \]
		 	
		 	where, ${\mu=\frac{r^{2}}{2r-1}}, ~\sigma^{2}=\frac{r^4 - 2r^3 + 2r^2 - r}{3(2r - 1)^2}$ and $\Phi$ is the cumulative distribution function of the Normal distribution $N(0,1).$
		 \end{lemma}

		 \begin{proof}
		 	We will solve this counting problem using the Central Limit Theorem. First, we define the set $A:=\{-r+1,\dots,-1,1,\dots,r\}$.

		 	Consider the inequality,
		 	\begin{equation} \label{eq:main}
		 		\sum_{i=1}^{n} |p_i| \leq x,
		 	\end{equation}
		 	where each \( p_i \in \mathbb{Z} \setminus \{0\} \), and \( -r < p_i \leq r \), i.e., \( p_i \in A \). Then, we first choose, $\lceil\frac{x}{r}\rceil \leq n\leq x. $
		 	
		 	Let \( X_i = |p_i| \in \{1, 2, \dots, r\} \), so the problem becomes
		 	\[
		 	\sum_{i=1}^n X_i \leq x.
		 	\]
		 	
		 	Now consider \( X_i \) as i.i.d. random variables uniformly distributed on \(A \). This implies we have two choices for $|p_i|$'s in $\{1,\dots,r-1\}$ and one choice for $|p_i|=r$, i.e. $\mathbb{P}(X_{i}=k)=\frac{2}{2r-1}$ for $k\in \{1,\dots,r-1\}$ and $\mathbb{P}(X_{i}=r)=\frac{1}{2r-1}$. Then
		 	\[
		 	\mathbb{E}[X_i] = \mu = \sum_{k=1}^{r-1} \frac{2}{2r-1}\cdot k +\frac{1}{2r-1}\cdot r
		 	\]
		 	\[=\frac{r(r-1)}{2r-1}+\frac{r}{2r-1}.\]
		 Therefore,	\[{\mu=\frac{r^{2}}{2r-1}}\]

		 	\[ \qquad \mathrm{Var}(X_i) = \sigma^2 = \mathbb{E}[X^{2}]-\mu^{2}=\sum_{k=1}^{r-1} \frac{2}{2r-1}\cdot k^{2} +\frac{1}{2r-1}\cdot r^{2}-\mu^{2}\]
		 	
		 	\[
		 	= \frac{2}{2r - 1} \cdot \frac{(r-1)r(2r - 1)}{6} + \frac{1}{2r - 1} r^2 -\mu^{2}= \frac{2(r - 1)r}{6} + \frac{r^2}{2r - 1}-\left( \frac{r^2}{2r - 1} \right)^2.
		 	\]
		 	\[
		 	\sigma^{2}= \frac{(2r - 1)^2 \cdot (r - 1)r + 3(2r - 1)r^2 - 3r^4}{3(2r - 1)^2}.
		 	\]
		 	Therefore,
		 	\[
		 	{
		 		\sigma^{2}=\frac{r^4 - 2r^3 + 2r^2 - r}{3(2r - 1)^2}.
		 	}
		 	\]

		 	By the Central Limit Theorem, the sum \( S_n = \sum X_i \) is approximately normally distributed. That means,
		 	\[
		 	\frac{S_n -n\mu}{\sqrt{n}\sigma}\approx {N}(0,1).
		 	\]
		 	So we have \[\lim_{n\rightarrow\infty}\mathbb{P}\left(\frac{S_n -n\mu}{\sqrt{n}\sigma}<~z\right)=\Phi(z)\]
		 	where $\Phi(z)$ is the cdf for the Normal distribution  $N(0,1).$ 
		 	
		 	Therefore, for $\epsilon=\frac{1}{x}{\Phi\left(\frac{x-x\mu}{\sqrt{x}\sigma}\right)}>0$ there is a natural number $K$ such that,
		 	
		 	\[\left|\mathbb{P}\left(\frac{S_n -n\mu}{\sqrt{n}\sigma}<~z\right)-\Phi(z)\right|<\frac{\epsilon}{2}\qquad \text{for all $n\geq K$}.\]
		 	
		 	We choose a sufficiently large $x$ such that $n\geq K$. We also choose $z=\frac{x-n\mu}{\sqrt{n}\sigma}$, then 
		 	\[
		 	\left|~\mathbb{P}\left(\frac{S_n -n\mu}{\sqrt{n}\sigma}\leq ~\frac{x-n\mu}{\sqrt{n}\sigma}\right)-\Phi\left(\frac{x-n\mu}{\sqrt{n}\sigma}\right )\right |<\frac{\epsilon}{2}
		 	\]
		 	\[
		 	\imp \left|\mathbb{P}(S_n \leq ~x)-\Phi\left(\frac{x-n\mu}{\sqrt{n}\sigma}\right)\right|<\frac{\epsilon}{2}. \]

		 	The total number of choices of an $n$-tuple $(X_i,)_{i=1}^{n}$ is $(2r-1)^{n}.$ Then,

	\[\left|\mathcal{N}_{n}(S_n \leq ~x)-(2r-1)^{n}\Phi\left(\frac{x-n\mu}{\sqrt{n}\sigma}\right)\right|<(2r-1)^{n}\frac{\epsilon}{2}\]
	\[\imp -\sum_{n=\lceil \frac{x}{r}\rceil}^{x}(2r-1)^{n}\frac{\epsilon}{2}<  \sum_{n=\lceil \frac{x}{r}\rceil}^{x}\left({\mathcal{N}_{n}(S_n \leq ~x)}-(2r-1)^{n}\Phi\left(\frac{x-n\mu}{\sqrt{n}\sigma}\right)\right)< \sum_{n=\lceil \frac{x}{r}\rceil}^{x}(2r-1)^{n}\frac{\epsilon}{2}.\]
	
	That implies
	\[
	- \frac{\sum_{n=\lceil \frac{x}{r}\rceil}^{x}(2r-1)^{n}\frac{\epsilon}{2}}{\sum_{n=\lceil \frac{x}{r}\rceil}^{x}(2r-1)^{n}\Phi(\frac{x-n\mu}{\sqrt{n}\sigma})}< \left( \frac{\sum_{n=\lceil \frac{x}{r}\rceil}^{x}\mathcal{N}_{n}(S_n \leq ~x)}{\sum_{n=\lceil \frac{x}{r}\rceil}^{x}(2r-1)^{n}\Phi(\frac{x-n\mu}{\sqrt{n}\sigma})}-1\right)< \frac{\sum_{n=\lceil \frac{x}{r}\rceil}^{x}(2r-1)^{n}\frac{\epsilon}{2}}{\sum_{n=\lceil \frac{x}{r}\rceil}^{x}(2r-1)^{n}\Phi(\frac{x-n\mu}{\sqrt{n}\sigma})}
	\]
	
	Since $\frac{x-n\mu}{\sqrt{n}\sigma}$ is a strictly decreasing function with respect to $n$ and $\Phi(z)$ is increasing, then $\Phi\left(\frac{x-x\mu}{\sqrt{x}\sigma}\right)$ is the smallest one among all other $\Phi(\frac{x-n\mu}{\sqrt{n}\sigma})$'s. Therefore,

	\[
	- \frac{\sum_{n=\lceil \frac{x}{r}\rceil}^{x}(2r-1)^{n}\frac{\epsilon}{2}}{\sum_{n=\lceil \frac{x}{r}\rceil}^{x}(2r-1)^{n}\Phi(\frac{x-x\mu}{\sqrt{x}\sigma})}< \left( \frac{\sum_{n=\lceil \frac{x}{r}\rceil}^{x}\mathcal{N}_{n}(S_n \leq ~x)}{\sum_{n=\lceil \frac{x}{r}\rceil}^{x}(2r-1)^{n}\Phi(\frac{x-n\mu}{\sqrt{n}\sigma})}-1\right)< \frac{\sum_{n=\lceil \frac{x}{r}\rceil}^{x}(2r-1)^{n}\frac{\epsilon}{2}}{\sum_{n=\lceil \frac{x}{r}\rceil}^{x}(2r-1)^{n}\Phi(\frac{x-x\mu}{\sqrt{x}\sigma})}
	\]
	
	\[
	- \frac{\frac{\epsilon}{2}}{\Phi(\frac{x-x\mu}{\sqrt{x}\sigma})}< \left( \frac{\sum_{n=\lceil \frac{x}{r}\rceil}^{x}\mathcal{N}_{n}(S_n \leq ~x)}{\sum_{n=\lceil \frac{x}{r}\rceil}^{x}(2r-1)^{n}\Phi(\frac{x-n\mu}{\sqrt{n}\sigma})}-1\right)<  \frac{\frac{\epsilon}{2}}{\Phi(\frac{x-x\mu}{\sqrt{x}\sigma})}
	\]
	Then we have,
	\[
	- \frac{1}{2x}< \left( \frac{\sum_{n=\lceil \frac{x}{r}\rceil}^{x}\mathcal{N}_{n}(S_n \leq ~x)}{\sum_{n=\lceil \frac{x}{r}\rceil}^{x}(2r-1)^{n}\Phi(\frac{x-n\mu}{\sqrt{n}\sigma})}-1\right)<  \frac{1}{2x}.
	\] 
	Now, after taking $x\rightarrow\infty$, we obtain that

	\[
	\lim_{x\rightarrow\infty}	\left|  \frac{\sum_{n=\lceil \frac{x}{r}\rceil}^{x}\mathcal{N}_{n}(S_n \leq ~x)}{\sum_{n=\lceil \frac{x}{r}\rceil}^{x}(2r-1)^{n}\Phi(\frac{x-n\mu}{\sqrt{n}\sigma})}-1\right|\rightarrow 0.
	\] 
	So, the number of the solutions for $\sum_{i=1}^{n}|p_{i}|\leq x$ with $\lceil \frac{x}{r}\rceil\leq n\leq x$ is,
	\[ {\sum_{n=\lceil \frac{x}{r}\rceil}^{x}\mathcal{N}_{n}(S_n \leq ~x)}\sim{\sum_{n=\lceil \frac{x}{r}\rceil}^{x}(2r-1)^{n}\Phi\left(\frac{x-n\mu}{\sqrt{n}\sigma}\right)}. \]
	
Now, if $n\leq \frac{x}{r}$ then any $n-$tuple of $X_{i}$'s will be part of the solutions of the Equation \eqnref{eq:main}. Then,

 \[{\sum_{n=1}^{\lfloor \frac{x}{r}\rfloor}\mathcal{N}_{n}(S_n \leq ~x)}=\sum_{n=1}^{\lfloor \frac{x}{r}\rfloor}(2r-1)^{n}.\]
 
 Therefore, we have,
 The approximation of the total number of solutions,
 \[\sum_{n=1}^{\lfloor \frac{x}{r}\rfloor}(2r-1)^{n}+{\sum_{n=\lceil \frac{x}{r}\rceil}^{x}(2r-1)^{n}\Phi\left(\frac{x-n\mu}{\sqrt{n}\sigma}\right)} .\]
  \end{proof}

	 Now, we will explore another similar lemma where, $p_{i}$ varies from $-r$ to $r$.
	 \begin{lemma}\label{4.3}
 Consider the following equation: 	
	 	\begin{equation}
	 		\sum_{i=1}^{n} |~p_{i}~|=S_{n}, 
	 	\end{equation}
	 	where, $p_{i}$, and $x$ are positive integers, and  \( -r \leq p_i \leq r \). Let $\mathcal{N}(S_{n}\leq x) $ denote the number of solutions of the above equation subject to the condition $S_n \leq x$.  Then,  \[ \mathcal{N}(S_{n}\leq x)\sim \sum_{n=1}^{\lfloor \frac{x}{r}\rfloor}(2r)^{n}+{\sum_{n=\lceil \frac{x}{r}\rceil}^{x}(2r)^{n}\Phi\left(\frac{x-n\mu}{\sqrt{n}\sigma}\right)} .\]
	 	
	 	where, ${\mu=\frac{r+1}{2}}, ~\sigma^{2}=\frac{r^2}{12}$ and $\Phi$ is the cumulative distribution function of the standard Normal distribution $N(0,1).$
	 \end{lemma}
	 \begin{proof} 	
	 	Let $B:=\{-r,-r+1,\dots,-1,1,\dots,r\}$. 	Consider the inequality,
	 	\begin{equation} \label{eq:main2}
	 		\sum_{i=1}^{n} |p_i| \leq x,
	 	\end{equation}
	 	where each \( p_i \in \mathbb{Z} \setminus \{0\} \), and \( -r \leq p_i \leq r \), i.e., \( p_i \in B \). Suppose $\lceil\frac{x}{r}\rceil \leq n\leq x. $
	 	
	 	Let \( X_i = |p_i| \in \{1, 2, \dots, r\} \). From \eqnref{eq:main2}, 
	 	\[
	 	\sum_{i=1}^n X_i \leq x.
	 	\]
	 	
	 	Now consider \( X_i \) as independent and identically distributed random variables uniformly distributed on \(B \). This implies,  $\mathbb{P}(X_{i}=k)=\frac{2}{2r}=\frac{1}{r}$ for $k\in \{1,\dots,r\}$. Then
	 	\[
	 	\mathbb{E}[X_i] = \mu = \sum_{k=1}^{r} \frac{1}{r}\cdot k 
	 	\]
	 
	 	Therefore,	\[{\mu=\frac{r+1}{2}}.\]

	 	Also,
	 	
	 	\[ \qquad \mathrm{Var}(X_i) = \sigma^2 = \mathbb{E}[X^{2}]-\mu^{2}=\sum_{k=1}^{r} \frac{1}{r}\cdot k^{2} -\mu^{2}\]
	 	
	 	\[
	 	= \frac{1}{r} \cdot \frac{(r+1)r(2r + 1)}{6}  -\mu^{2}= \frac{(2r+ 1)(r+1)}{6} -\left( \frac{r+1}{2} \right)^2.
	 	\]

    It follows 	 	\[
	 	{
	 		\sigma^{2}=\frac{r^{2}-1}{12}.
	 	}
	 	\]

	 The total number of choices of the $n$-tuples  $(X_i)_{i=1}^{n}$ is $(2r)^{n}.$ Then
	 	by continuing the same process as the previous lemma, we obtain that, 
	 	 	the number of the solutions for $\sum_{i=1}^{n}|p_{i}|\leq x$ with $\lceil \frac{x}{r}\rceil\leq n\leq x$ is,
	 	\[ {\sum_{n=\lceil \frac{x}{r}\rceil}^{x}\mathcal{N}_{n}(S_n \leq ~x)}\sim{\sum_{n=\lceil \frac{x}{r}\rceil}^{x}(2r)^{n}\Phi\left(\frac{x-n\mu}{\sqrt{n}\sigma}\right)}. \]
	 	
	 	Now, if $n\leq \frac{x}{r}$ then any $n-$tuple of $X_{i}$'s will be part of the solutions of the Equation \eqnref{eq:main2}. Then,
	 	
	 	\[{\sum_{n=1}^{\lfloor \frac{x}{r}\rfloor}\mathcal{N}_{n}(S_n \leq ~x)}=\sum_{n=1}^{\lfloor \frac{x}{r}\rfloor}(2r)^{n}.\]
	 	
	 	Therefore, we have,
	 	The approximation of the total number of solutions,
	 	\[\sum_{n=1}^{\lfloor \frac{x}{r}\rfloor}(2r)^{n}+{\sum_{n=\lceil \frac{x}{r}\rceil}^{x}(2r)^{n}\Phi\left(\frac{x-n\mu}{\sqrt{n}\sigma}\right)} .\]
        This proves the lemma. 
	 \end{proof}
	 
		Let $p$ be an odd natural number for the Hecke group $\Gamma_{p}$. Then the reduced reciprocal word $\alpha$ in that group has only one type of form $ \io\g^{k_{1}}\io \g^{k_{2}}\dots \io \g^{k_{n}}\io \g^{-k_{n}}\dots\io \g^{-k_{2}}\io \g^{-k_{1}} $ (see \lemref{lem}). Now, we have :
	
		\subsection{Proof of the \thmref{main1}}
		Let $\alpha$ be the cyclically reduced reciprocal word, and it is of the form $ \io\g^{k_{1}}\io \g^{k_{2}}\dots \io \g^{k_{n}}\io \g^{-k_{n}}\dots\io \g^{-k_{2}}\io \g^{-k_{1}} .$
			From the previous section, we know that \[\sum_{i=1}^{n}2\left(\left|k_{i}\right|+1\right)=S_{n}.\]
			We consider every $X_{i}=2\left(\left|k_{i}\right|+1\right)$ as independent and identically distributed on $B=\{\pm r,\dots, \pm2,\pm 1\}.$ To find the asymptotic value of the growth of the number of reciprocal geodesics, it is suffices to find the growth of the number of the solutions of the following inequality,
            \[ S_n\leq x\]
            where $n$ is not fixed.
		Now, the main problem becomes $$S_{n}=\sum_{i=1}^{n}X_{i}\leq x.$$ The process of this lemma is almost similar to \lemref{4.3}, but here the random variables are different, and we should choose $n$ from $\lceil\frac{x}{2(r+1)}\rceil$ two $\lfloor\frac{x}{4}\rfloor$ , unlike the previous lemma.  We need to take the mean and the variance accordingly. 
		Therefore, the mean, i.e.	\[
		\mathbb{E}[X_i] = \mu = \sum_{i=1}^{n}2(\left|k_{i}\right|+1)\mathbb{P}(X_{i}=2\left|p_{i}\right|+2)=\sum_{k=1}^{r} \frac{1}{r}\cdot (2k +2)=\frac{r(r+1)+2r}{r}=r+3.
		\]
			Now, the variance, i.e.
			
			\begin{align*}
				\mathrm{Var}(X_i) = \sigma^2 = \mathbb{E}[X^{2}]-\mu^{2}
				&= \frac{1}{r}\sum_{k=1}^r (2k+2)^2 - (r+3)^2 \\[6pt]
				&= \frac{1}{r}\Biggl(4\sum_{k=1}^r k^2 + 8\sum_{k=1}^r k + 4r\Biggr) - (r+3)^2 \\[6pt]
				&= \frac{1}{r}\Biggl(4\frac{r(r+1)(2r+1)}{6} + 8\frac{r(r+1)}{2} + 4r\Biggr) - (r^2+6r+9) \\[6pt]
				&= \frac{4r^3 + 18r^2 + 26r}{3r} - (r^2+6r+9) \\[6pt]
				&= \frac{4r^2 + 18r + 26}{3} - (r^2+6r+9) \\[6pt]
				&= \frac{4r^2 + 18r + 26 - 3r^2 - 18r - 27}{3} \\[6pt]
				&= \frac{r^2 - 1}{3}.
			\end{align*}
			 Then, the number of the reciprocal classes in Hecke groups with at most length $x$ is asymptotic to
			 	\[\frac{1}{2}\sum_{n=1}^{\lfloor \frac{x}{2(r+1)}\rfloor}(2r)^{n}+\frac{1}{2}{\sum_{n=\lceil \frac{x}{2(r+1)}\rceil}^{\lfloor \frac{x}{4}\rfloor}(2r)^{n}\Phi\left(\frac{x-n\mu}{\sqrt{n}\sigma}\right)} .\] Since we have exactly two elements in that form for every reciprocal class, we have multiplied $\frac{1}{2}$.			 \qed
	
    \medskip
		In case of even $p$, we have different types of reciprocal classes. In the following theorem, we describe the counting problem on symmetric reciprocal classes.
	
	\begin{lemma}\label{4.6}
	For even $p=2r $, let $Sym_{x}$ be the set of symmetric reciprocal classes in the Hecke group $\Gamma_{p}$, those have at most $x$ word lengths. Then,
	\[\left|Sym_{x}\right|\sim 	\frac{1}{2}\sum_{n=1}^{\lfloor \frac{x}{2(r+1)}\rfloor}(2r-1)^{n}+\frac{1}{2}{\sum_{n=\lceil \frac{x}{2(r+1)}\rceil}^{\lfloor\frac{x}{4}\rfloor}(2r-1)^{n}\Phi\left(\frac{x-n\mu}{\sqrt{n}\sigma}\right)} \] where, $\mu=\frac{2r^2 + 4r - 2}{2r - 1}$ and $\sigma^{2}=\frac{16r^3 + 36r^2 + 32r - 12}{6(2r - 1)}.$
\end{lemma}			
		
		\begin{proof}
			In the case of even $p$, the reciprocal reduced words are also of the form of $\alpha$ i.e., $ \io\g^{k_{1}}\io \g^{k_{2}}\dots \io \g^{k_{n}}\io \g^{-k_{n}}\dots\io \g^{-k_{2}}\io \g^{-k_{1}} $ with all $k_{i}\in \{-r+1,-r+2,\dots,-1,1,\dots,r\}.$ 
			Let $X_{i}=2\left(\left|k_{i}\right|+1\right)$  for some $i$, be an independent and identically distributed random variable uniformly distributed on $\{r,\pm (r-1),\dots, \pm2,\pm 1\}$.  Now we have the mean, i.e.,
		
			\begin{align*}
				\mathbb{E}[X_i] = \mu = \sum_{i=1}^{n}2(\left|k_{i}\right|+1)\mathbb{P}(X_{i}=2\left|k_{i}\right|+2)
			   &=	\sum_{k=1}^{r-1} \frac{2}{2r-1}(2k + 2) + \frac{2r + 2}{2r - 1}\\[6pt]
				&= \frac{2}{2r - 1} \sum_{k=1}^{r-1}(2k + 2) + \frac{2r + 2}{2r - 1} \\[6pt]
			&= \frac{2}{2r - 1} \left(2\sum_{k=1}^{r-1}k + 2(r - 1)\right) + \frac{2(r + 1)}{2r - 1} \\[6pt]
				&= \frac{2}{2r - 1} \left(2 \cdot \frac{(r-1)r}{2} + 2(r - 1)\right) + \frac{2(r + 1)}{2r - 1} \\[6pt]
				&= \frac{2}{2r - 1} \left(r(r - 1) + 2(r - 1)\right) + \frac{2(r + 1)}{2r - 1} \\[6pt]
				&= \frac{2}{2r - 1} \left((r + 2)(r - 1)\right) + \frac{2(r + 1)}{2r - 1} \\[6pt]
				&= \frac{2(r + 2)(r - 1) + 2(r + 1)}{2r - 1} \\[6pt]
				&= \frac{2(r^2 + r - 2) + 2r + 2}{2r - 1} \\[6pt]
				&= \frac{2r^2 + 2r - 4 + 2r + 2}{2r - 1} \\[6pt]
				&= \frac{2r^2 + 4r - 2}{2r - 1}.
 			\end{align*}
			And the variance, i.e.,
			\begin{align*}
					\mathrm{Var}(X_i) = \sigma^2 = \mathbb{E}[X^{2}]-\mu^{2}
				&= \frac{2}{2r-1}\sum_{k=1}^{r-1} (2k+2)^2+\frac{(2r + 2)^2}{2r - 1} - \frac{2r^2 + 4r - 2}{2r - 1} \\[6pt]
	           &= \frac{1}{2r - 1} \left[ 2\sum_{k=1}^{r-1}(2k + 2)^2 + (2r + 2)^2 - (2r^2 + 4r - 2) \right] \\[6pt]
				&= \frac{1}{2r - 1} \left[ 2\cdot 4\sum_{k=1}^{r-1}(k^2 + 2k + 1) + 4r^2 + 8r + 4 - 2r^2 - 4r + 2 \right] \\[6pt]
				&= \frac{1}{2r - 1} \left[ 8 \left( \sum_{k=1}^{r-1}k^2 + 2\sum_{k=1}^{r-1}k + (r - 1) \right) + 2r^2 + 4r + 6 \right] \\[6pt]
                \end{align*}
              That implies,  
                 \[ \sigma^2 = \frac{1}{2r - 1} \left[ 8\left( \frac{(r - 1)r(2r - 1)}{6} + (r - 1)r + (r - 1) \right) + 2r^2 + 4r + 6 \right] \]
				\[= \frac{1}{2r - 1} \left[ \frac{8(r - 1)(2r^2 + 5r + 6)}{6} + 2r^2 + 4r + 6 \right]\] 
				\[= \frac{1}{2r - 1} \left[ \frac{16r^3 + 24r^2 + 8r - 48}{6} + 2r^2 + 4r + 6 \right]\] 
				\[= {\frac{16r^3 + 36r^2 + 32r - 12}{6(2r - 1)}}.\]
			
			Continuing as \lemref{even} and the previous theorem, we have the asymptotic value: 
				\[\left|Sym_{x}\right|\sim 	\frac{1}{2}\sum_{n=1}^{\lfloor \frac{x}{2(r+1)}\rfloor}(2r-1)^{n}+\frac{1}{2}{\sum_{n=\lceil \frac{x}{2(r+1)}\rceil}^{\lfloor\frac{x}{4}\rfloor}(2r-1)^{n}\Phi\left(\frac{x-n\mu}{\sqrt{n}\sigma}\right)}. \]

				$ \frac{1}{2}$ appears because of having two reduced form of $\alpha$ in that reciprocal class.
		\end{proof}
		
		\medspace
		
		 We can similarly calculate the number of $ p$-reciprocal classes as well as symmetric $p $-reciprocal classes. The power series form of $|Sym_{x}|$  ends at $\lfloor\frac{x}{4}\rfloor$th term. The power series ends at $\lfloor\frac{x-q}{4}\rfloor$ where, $q=2r+2,~r+1$ respectively for $p$-reciprocal classes and symmetric $p$-reciprocal classes. So, the growth of asymptotic functions in that case does not surpass  the growth of the number of the symmetric reciprocal classes. Let $Rec_{x}$, $SR_{x}$ be the sets of $p$-reciprocal and symmetric $p$-reciprocal classes respectively, in the Hecke group $\Gamma_{p}$, which have at most $x$ word lengths. Also, we assume that the set of all reciprocal classes in 
        $\Gamma_p$ with word length up to $x$ is $T_{x}$.
		 Then, any large $x,$ \[|Sym_{x}|\leq |T_{x}|\leq |Sym_{x}|+|Rec_{x}|+|SR_{x}|.\]
		 We know that there are other types of reciprocal classes except symmetric reciprocal classes. That justifies the first relation. And the second relation is also satisfied since there are some reciprocal classes that are in common. For example, $U=\io\gt\g^{k_{1}}\dots \io \g^{k_{n}}\io \g^{-k_{n}}\dots\io \g^{-k_{1}} \io\gt\g^{k_{1}}\dots \io \g^{k_{n}}\io \g^{-k_{n}}\dots\io \g^{-k_{1}}$, is symmetric reciprocal as well as $p$-reciprocal and it is a power of symmetric $p $-reciprocal element. Thus, we have,
		 \[|Sym_x|\leq |{T_{x}}|\precsim |Sym_x|.\]
		 As the highest powers of $(2r-1)$ in $Rec_{x},~SR_{x}$ are less than $\lfloor\frac{x}{4}\rfloor$, then we have 
          \[ |T_x|\simeq |Sym_x|\]
	
		\subsection{Proof of \thmref{main2}}

			The theorem follows from the \lemref{4.6}, and from the fact $ |{T_{x}}|\simeq{|Sym_{x}|}.$ \qed

            Now, we will discuss about the asymptotic growth of the number of primitive and non-primitive reciprocal classes. Before that, we have the following lemma:
	            \begin{lemma}\label{7.1}
	            	Let $W_{x}$ and $W^{np}_x$ be the collection of reciprocal classes and non-primitive reciprocal classes in $\Gamma_{p}$, those have word length at most $x$ , respectively. Then
	            	 \[\left|W^{np}_{x}\right|\leq 2x(x+1)\left|W_{\lfloor\frac{x}{2}\rfloor}\right|.\]
	            	
	            \end{lemma}\begin{proof}
	             Let $R^p_{x}$ and $R^{np}_x$ be the collection of primitive and non-primitive reciprocal classes in $\Gamma_{p}$, those have word length exactly $x$ respectively. Now we will use the same technique as \cite{DG2} to get an upper bound for non-primitive reciprocal geodesics.
	          From the hypothesis, we know that \[\left|W^{np}_{x}\right|=\sum_{s=1}^{x} \left|R^{np}_{s}\right|~~\text{and }~~\left|W_{x}\right|=\sum_{s=1}^{x}\left|R_{x}\right|\] where $W_{x}$ and  $W^{np}_{x}$ are the collection of reciprocal classes and the non-primitive reciprocal classes in $\Gamma_{p}$, and those have word lengths at most $x$ . Then we have
	          \[\left|W^{np}_{x}\right|=\sum_{s=1}^{x} \left|R^{np}_{s}\right|\leq \sum_{s=1}^{x}\sum_{d|s}\left|R^{p}_{d}\right|\leq \sum_{s=1}^{x}\sum_{d=1}^{\lfloor\frac{s}{2}\rfloor}\left|R_{d}\right|\leq \sum_{s=1}^{x}\frac{s}{2}\left|R_{\lfloor\frac{s}{2}\rfloor}\right|.\]
	          So, by Cauchy-Schwartz's inequality, we have
	          \[\left|W^{np}_{x}\right|\leq\left(\sum_{s=1}^{x}\frac{s}{2}\right)\left(\sum_{s=1}^{x}\left|R_{\lfloor\frac{s}{2}\rfloor}\right|\right)\leq \frac{x(x+1)}{4}(|R_{1}|+|R_{1}|+|R_{2}|+|R_{2}|+\dots+|R_{\lfloor\frac{x}{2}\rfloor}|+|R_{\lfloor\frac{x}{2}\rfloor}|).\]
	          Then implies,
	          \[\left|W^{np}_{x}\right|\leq \frac{x(x+1)}{2}\sum_{s=1}^{\lfloor\frac{x}{2}\rfloor}\left|R_{s}\right|=\frac{x(x+1)}{2}\left|W_{\lfloor\frac{x}{2}\rfloor}\right|.\]
	           Hence, this proves the lemma. \end{proof}
	          \begin{lemma}\label{7.2}
	          		Let $W_{x}$ be the collection of the reciprocal classes in $\Gamma_{p}$, those have word length at most $x$ , respectively. Then for sufficiently large $x$, there exists a constant $C$ such that
	          		\[ 2\left|W_{\lfloor\frac{x}{2}\rfloor}\right|^{2}\leq C^2\left|W_{x}\right|. \]
	          \end{lemma}
	          \begin{proof}
	          Let $w$ be the cyclically reduced reciprocal word, and it is of the form\\ $ \io\g^{k_{1}}\io \g^{k_{2}}\dots \io \g^{k_{n}}\io \g^{-k_{n}}\dots\io \g^{-k_{2}}\io \g^{-k_{1}} $ , and the word length of $w$ is less than $x.$  That means,
	          \[\sum_{i=1}^{n}2\left(\left|k_{i}\right|+1\right)\leq x,\qquad -r<k_{i}\leq r.\]
	          The number of solutions of the inequality is $2\left|Sym_{x}\right|$, where $Sym_{x}$ is the collection of symmetric reciprocal classes in $\Gamma_{p}$ with word length at most $x$. 
	          Now we consider the solution sets of $(k_{i})$'s, which can be divided into two parts.
	          We consider an event $H$ to have occurred if the solution sets of each part satisfy this inequality,
	          \[\sum_{i}2(\left|x_{i}\right|+1)\leq \frac{x}{2},\qquad -r<x_{i}\leq r..\]
	          So, we have the probability of occurring $H$ , i.e.,
	          \[\mathbb{P}(H)= \frac{\text{The number of times $H$ occurs}}{\text{The number of all the solutions}}\]
	          \[\imp \mathbb{P}(H)=\frac{2\left|Sym_{\lfloor\frac{x}{2}\rfloor}\right| 2\left|Sym_{\lfloor\frac{x}{2}\rfloor}\right|}{2\left|Sym_{x}\right|}\leq 1.\]
	        
	          Also, we know that for sufficiently large $x$, there exists a non-zero constant $C$ such that $$\left|Sym_x\right|\leq \left|W_{x}\right|\leq C \left|Sym_{x}\right|.$$
	         This implies,
	       \[2\frac{\left|W_{\lfloor\frac{x}{2}\rfloor}\right|\left|W_{\lfloor\frac{x}{2}\rfloor}\right|}{C^2\left|W_{x}\right|}\leq 2\frac{\left|Sym_{\lfloor\frac{x}{2}\rfloor}\right|\left|Sym_{\lfloor\frac{x}{2}\rfloor}\right|}{\left|Sym_{x}\right|}\leq 1.\]
	          This proves the lemma.
	          \end{proof}

	          \medskip
	          \subsection{Proof of the \corref{cormain}} 
	          By \lemref{7.1}, we have
	          \[
	          |W_{x}| = |W^{np}_{x}| + |W^{p}_{x}| \leq\frac{x(x+1)}{2}\left|W_{\lfloor\frac{x}{2}\rfloor}\right|+|W^{p}_{x}|.
	          \]
	          We obtain the inequality
	          \[
	          |W_{x}| - \frac{x(x+1)}{2}\left|W_{\lfloor\frac{x}{2}\rfloor}\right|\leq|W^{p}_{x}| \leq|W_{x}|.
	          \]
	          Now we get,
	          \[
	          1 - \frac{x(x+1)}{2}\frac{\left|W_{\lfloor\frac{x}{2}\rfloor}\right|}{\left|W_{x}\right|}
	          \leq \frac{|W^{p}_{x}|}{|W_{x}|} \leq 1.
	          \]
	          
	         Since by \lemref{7.2}, we have obtian
	         \[
	         \frac{\left|W_{\lfloor\frac{x}{2}\rfloor}\right|}{\left|W_{x}\right|}\leq \frac{C^2}{2|W_{\lfloor\frac{x}{2}\rfloor}|}
	         \]
	         \[\imp  \frac{x(x+1)}{2}\frac{\left|W_{\lfloor\frac{x}{2}\rfloor}\right|}{\left|W_{x}\right|}\leq C^2\frac{x(x+1)}{4|W_{\lfloor\frac{x}{2}\rfloor}|} \]
	         By taking  \( x \to \infty \), the ratio $\frac{x(x+1)}{|W_{\lfloor\frac{x}{2}\rfloor}|}$ tends to zero as $|W_{x}|$ bounded below by an exponential function of $x$. Thus,
	          \[
	          \frac{|W^{p}_{x}|}{|W_{x}|}\to 1 \quad \text{as } x \to \infty,
	          \]
	          which implies
	          \[
	        {|W^{p}_{x}|}\sim{|W_{x}|}.
	          \]
	          This proves the corollary.\qed

		\subsection{Proof of \corref{1.4}}
			In the particular case of the modular group (i.e., $\Gamma_3$), we have $r=1$, then $\frac{x}{r+1}=\frac{x}{2}.$ That means, the term with the cumulative distribution function
			 of the normal distribution will not appear. Therefore, the number of the reciprocal classes in the modular group with at most length $x$ is asymptotic to
			\[\frac{1}{2}\sum_{n=1}^{\lfloor \frac{x}{4}\rfloor}2^{n}=(2^{\lfloor \frac{x}{4}\rfloor}-1)\sim 2^{\lfloor \frac{x}{4}\rfloor}.  \] Also, the growth of the number of reciprocal classes is the same as the growth of the number of primitive reciprocal classes by \corref{cormain}. That reflects the required result. \qed
		
		
	
		\subsection*{Acknowledgment} 
		Das acknowledges support from UGC-NFSC senior research fellowship. Gongopadhyay is partially supported by the SERB core research grant CRG/2022/003680.
		

	\end{document}